\documentclass[11pt]{amsart}
\usepackage[T1]{fontenc}
\usepackage[utf8]{inputenc}
\usepackage[english]{babel}
\usepackage{amsmath, amsthm, amssymb, amsfonts, enumerate, mathrsfs, dsfont, comment, bm, mathtools}
\usepackage{algorithm}
\usepackage{algpseudocode}
\usepackage{enumitem}

\usepackage{color, geometry, todonotes, indentfirst, graphicx, pdfpages}
\usepackage{fancyhdr}

\usepackage{appendix}

\usepackage{etoolbox}
\makeatletter
\def\section{\@startsection{section}{1}%
  \z@{2.5ex plus 1ex minus .2ex}{1.5ex plus .2ex}%
  {\centering\normalfont\Large\scshape}}
\makeatother

\usepackage[colorlinks=true,linkcolor=blue,urlcolor=blue]{hyperref}

\allowdisplaybreaks

\newtheorem{theorem}{Theorem}[section]
\newtheorem{remark}[theorem]{Remark}
\newtheorem{assumption}[theorem]{Assumption}
\newtheorem{lemma}[theorem]{Lemma}
\newtheorem{proposition}[theorem]{Proposition}

\newtheorem{definition}[theorem]{Definition}

\theoremstyle{plain}

\def \R{\mathbb{R}}
\def \N{\mathbb{N}}

\def \E{{\mathbb{E}}}

\def \and{\quad \text{and} \quad}

\DeclareMathOperator*{\argmax}{arg\,max}

\DeclareMathOperator*{\essinf}{ess\,inf}

\numberwithin{equation}{section}

 \newcommand{\F}{\mathcal{F}}
\newcommand{\G}{\mathcal{G}}
\newcommand{\pr}{\mathbb{P}}
\newcommand{\abs}[1]{\left\lvert{#1}\right\rvert}

\newcommand{\tonde}[1]{\left({#1}\right)}
\newcommand{\quadre}[1]{\left[{#1}\right]}
\newcommand{\graffe}[1]{\left\lbrace{#1}\right\rbrace}
\newcommand{\norm}[1]{\left\lVert{#1}\right\rVert}

\DeclareMathAlphabet{\mathbbold}{U}{bbold}{m}{n}
\newcommand{\condexp}[2]{\mathbb{E}\left[ {#1} \,\middle\vert \,{#2}\right]}
\newcommand{\expect}[1]{\mathbb{E} \left[ {#1} \right]}
\newcommand{\tildeexpect}[1]{\tilde{\mathbb{E}} \left[ {#1} \right]}
\newcommand{\tildecondexp}[2]{ \tilde{\mathbb{E}}\left[ {#1} \,\middle\vert \,{#2}\right]}

\newcommand{\indicator}[1]{\mathbbold{1}_{ {#1} }}
\newcommand{\randprobspace}{\mathbb{L}^0_{\G}\,(\Omega,\mathcal{P}(E))}
\newcommand{\randproboptstopp}{\mathbb{L}^0_{\G}\,(\,\Omega,\mathcal{P}( [0,T] ))}

\newcommand{\pushright}[1]{\ifmeasuring@#1\else\omit\hfill$\displaystyle#1$\fi\ignorespaces}
\newcommand{\mail}[1]{\href{mailto:#1}{\texttt{#1}}}

\title[mean-field Games of Optimal Stopping with Common Noise]{Existence of Strong Randomized Equilibria in Mean-Field Games of Optimal Stopping with Common Noise}

\author[G.~Ferrari]{Giorgio~Ferrari\textsuperscript{\MakeLowercase{a},1}}
\thanks{\noindent\textsuperscript{a} Bielefeld University, Center for Mathematical Economics (IMW), Bielefeld (Germany).}
\author[A.~Pajola]{Anna~Pajola\textsuperscript{\MakeLowercase{a},2}}
\thanks{\noindent
\noindent \textsuperscript{1} E-mail: \mail{giorgio.ferrari@uni-bielefeld.de}.
\\
\noindent \textsuperscript{2} E-mail: \mail{anna.pajola@uni-bielefeld.de}.
}

\date{\today}

\numberwithin{equation}{section}

\begin{document}


\begin{abstract}
We study a mean-field game of optimal stopping and investigate the existence of strong solutions via a connection with the Bank-El Karoui's representation problem. Under certain continuity assumptions, where the common noise is generated by a countable partition, we show that a strong randomized mean-field equilibrium exists, in which the mean-field interaction term is adapted to the common noise and the stopping time is randomized. Furthermore, under suitable monotonicity assumptions and for a general common noise, we provide a comparative statics analysis of the set of strong mean-field equilibria with strict equilibrium stopping times.
\end{abstract} 

\maketitle

\smallskip 
{\textbf{Keywords}:}
mean-field game of optimal stopping, randomized stopping time, common noise, Bank-El Karoui's representation theorem

\smallskip 
{\textbf{AMS subject classification}}: 
91A16, 
60G40,  
93E20, 
91A55 

 
\section{ Introduction}
\label{sec:intro}

Introduced independently in \cite{huang} and \cite{lasry}, mean-field stochastic games serve as approximation models for $N$-player stochastic games as $N$ tends to infinity. In mean-field games, a large number of exchangeable agents play a symmetric game with mean-field interactions: each player responds to the distribution of other agents’ actions rather than to the actions of individual agents. This distribution is commonly referred to as the mean-field interaction term. We refer to the two-volume monography \cite{carmona2018probabilistic} for theory and applications.

Optimal stopping problems feature prominently in the literature, with applications ranging from the optimal exercise of American options to the optimal exit time of firms from a market. While not an exhaustive review, we refer to \cite{chow1991theory},\cite{dynkin1963optimum}, \cite{karoui1991new}, \cite{peskir2006optimal} and \cite{shiryaev2011optimal}. However, mean-field games of optimal stopping have received less attention compared to mean-field games with regular controls.

This paper investigates a mean-field game of optimal stopping with common noise, in which a representative agent seeks the optimal stopping time to maximize a reward functional. Both the running and terminal reward functions depend on the mean-field interaction term, which, in equilibrium, corresponds to the conditional law of the optimal stopping time given the common noise. The setting we consider is non-Markovian, as the reward functions are general random functions, and the analysis is performed through a purely probabilistic approach. We seek strong mean-field equilibria in the sense of strong solutions to stochastic differential equations: we fix the probability space and the $\sigma$-algebra representing the common noise and look for adapted solutions. As it is well known in the literature, the presence of the common noise complicates the analysis, as it enlarges the space of mean-field interaction terms from deterministic probability measures to random probability measures. This poses challenges in proving the existence of strong equilibria. In mean-field games without common noise, strong solutions are usually obtained using fixed point theorems, such as Schauder's or Kakutani's theorems. Both these results rely on the ability to identify compact subsets within the space of mean-field interactions. In the setting of deterministic probability measures, equipped with the topology induced by the weak convergence, this can be achieved through tightness arguments and Prokhorov's theorem. However, the situation becomes more complex in the presence of common noise, as the space of random probability measures is considerably larger. As noted in \cite{carmona2016mean}, this increased complexity, along with the discontinuity of conditioning with respect to a general common noise, prevents the use of most topological fixed point theorems. To address this, with the aim of proving existence of strong equilibria, we assume that the common noise is generated by a countable partition of the probability space $\Omega$. Under this assumption, the space of random probability measures can be identified with the product of countable many copies of the space of deterministic probability measures and hence it is an Hausdorff locally convex topological vector space, as required by Schauder's fixed point theorem. Furthermore, the assumption significantly simplifies the analysis by enabling the use of tightness arguments to identify the compact subsets necessary to applying Schauder's theorem. Notice that requiring the common noise to be generated by a countable partition of $\Omega$ excludes the case of a Brownian common shock but includes, e.g., the case of the tail $\sigma$-algebra generated by an irreducible, recurrent Markov chain (see also Remark \ref{remark_commonnoise} below).

Our main contribution is an existence result for strong randomized mean-field equilibria in a setting with continuity assumptions of the reward functions with respect to the interaction terms.\ We define a strong randomized mean-field equilibrium as a pair, in which the mean-field interaction term is adapted to the common noise, while the stopping time is randomized, following the framework introduced in \cite{baxter_chacon}. In this sense, we allow additional randomization in the stopping times, while maintaining adapted mean-field interaction terms. This trade-off provides compactness for both the stopping times and the interaction terms, which is essential to prove the existence of solutions, while still preserving their strong nature.

More precisely, we build upon the results of \cite{mf-bek}, which leverage the connection between optimal stopping and the Bank-El Karoui's representation problem to prove the existence of strong $\varepsilon$-mean-field equilibria (cf. Proposition 2.15 in \cite{mf-bek}). Notice that in \cite{mf-bek}
strong equilibria (i.e. with $\varepsilon=0$) cannot be established because the optimal stopping times -- characterized as hitting times of the running supremum $\hat{L}$ of the solution $L$ to a suitably defined Bank-El Karoui's representation problem -- are not continuous functions of $\hat{L}$ with respect to the Lévy distance, since $\hat{L}$ is not strictly increasing. Hence, to recover the continuity needed for the application of Schauder's fixed point theorem, one must instead consider the hitting times of a strictly increasing transformation of $\hat{L}$, namely $\hat{L}+\varepsilon\,\text{id}$. As a consequence, the resulting stopping times are no longer optimal and one obtains strong $\varepsilon$-equilibria rather than strong equilibria.

The aim is then to take suitable limit as $\varepsilon$ tends to zero of the strong $\varepsilon$-equilibria determined in \cite{mf-bek}. To accomplish that, we relax the problem by allowing randomized stopping times as admissible strategies, analogous to the shift from pure to mixed strategies in classical game theory (cf. \cite{aumann}). In particular, as in \cite{baxter_chacon}, we define randomized stopping times as stopping times on the enlarged probability space $\Omega\times [0,1]$.\ These can be interpreted as a mechanism that, for each real number $v\in[0,1]$, yields a corresponding stopping time $\tau_v$. Notably, the space of randomized stopping times,  equipped with the topology induced by the Baxter-Chacon convergence (see \cite{baxter_chacon}), is compact, a property widely used in the literature (see \cite{coquet}, \cite{laraki2005value} and \cite{touzi2002continuous}). We refer to \cite{touzi2002continuous} also for a comprehensive overview of various definitions of randomized stopping times. Thanks to the compactness of both the space of randomized stopping times and the space of interaction terms, we can extract a convergent subsequence from the sequence of strong $\varepsilon$-equilibria established in \cite{mf-bek} as $\varepsilon\to 0$.\ We then show that the limit of this subsequence constitutes a strong randomized mean-field equilibrium, leveraging the continuity of the reward functional with respect to the pair formed by a randomized stopping time and a mean-field interaction term.  In order to ensure this property, we impose the additional assumption that the reward functions depend on the interaction term $m$ only through a deterministic function of an auxiliary continuous stochastic process $X^m$. Remarkably, the running reward function is allowed to depend on the paths of $X^m$, not just its current value, so the setting is still non-Markovian.

In addition, we study the mean-field game of optimal stopping in a setting with an order structure and monotonicity properties of the reward functions with respect to the mean-field interaction terms.\ In this framework, the common noise is represented by a general $\sigma$-algebra, not necessarily generated by a countable partition of $\Omega$, here also allowing for a Brownian common shock.\ We establish the existence of strong mean-field equilibria (with strict optimal stopping times, not randomized) by applying Tarski's fixed point theorem, a result which appears in earlier works (e.g.,  \cite{carmona2017mean}, \cite{dianetti2023unifying}, \cite{mf-bek} and \cite{possamai2023mean}). Our contribution lies in a comparative statics analysis of the set of strong mean-field equilibria.\ 
More precisely, we consider two different sets of terminal and running reward functions and introduce an ordering between them.\ We then exploit the characterization of the largest and smallest optimal stopping times as hitting times of the running supremum process of the solution $L$ to an auxiliary Bank-El Karoui's representation problem, as presented in \cite{mf-bek}. The ordering of the reward functions induces an ordering on the corresponding solutions $L$, which in turn leads to an ordering of the hitting times.\ This allows us to establish an order between the extremal elements of the sets of strong mean-field equilibria.

As already mentioned, the literature on mean-field games of optimal stopping is still relatively limited.\ In \cite{carmona2017mean}, the authors studied a game of optimal stopping with Brownian common noise, inspired by a model of bank runs. They proved the existence of strong mean-field equilibria in a setting with strategic complementarity, and weak existence under continuity assumptions.\ Among existing works, this one is closest to ours in terms of the definition of equilibrium.\ Additionally, \cite{nutz2018mean} considers a game where the interaction depends on the number of players who have already stopped.\ \cite{bertucci2018optimal} adopts a purely analytical approach -- through the study of coupled Hamilton-Jacobi-Bellman and Fokker-Planck equations associated to an equilibrium -- and proves the existence of mixed solutions.\ By contrast, \cite{bouveret2020mean} and \cite{dumitrescu2021control} take a probabilistic approach, using linear programming to prove the existence of relaxed solutions in the sense of occupational measures, replacing stopping times.\ These techniques have been applied to a specific optimal stopping game with common noise in the context of electricity markets in \cite{aid2020entry}. Finally, \cite{possamai2023mean} studies such games through the master equation framework, while \cite{yu2025major} solves a discrete-time major-minor mean-field game combining linear programming techniques and entropy regularization.

The rest of the paper is organized as follows. In Section \ref{section_problem} we provide the formulation of the mean-field game of optimal stopping and the definition of strong mean-field equilibria. Then two distinct settings are considered. We investigate the problem under continuity conditions in Section \ref{section_optstopp_cont} and under monotonicity conditions in Section \ref{section_optstopp_mon}. In Appendix \ref{appendix_b} we provide a brief summary of the results in \cite{mf-bek} needed in our analysis, while in Appendix \ref{appendix_a} we provide the proofs of some auxiliary results stated in the previous sections.

\smallskip \section{  Formulation of the problem }
\label{section_problem}

\subsection{Preliminaries and notation}
Let $T\in[0,+\infty]$ be the time horizon and $(\Omega, \F, \mathbb{F}, \pr)$ a filtered probability space with filtration $\mathbb{F}:=(\F_t)_{t\in[0,T]}$ satisfying the usual conditions of right-continuity and completeness. We assume that the space $\Omega$ is rich enough to support an additional $\sigma$-algebra $\mathcal{G}$, which will represent the common noise. 
    We denote by $\mathcal{S}$ the collection of all $\mathbb{F}$-stopping times with values in $[0,T]$ a.s.

    \medskip Let $E$ be a non-empty Polish space and $\mathcal{P}(E)$ the space of all Borel probability measures on $E$, equipped with the topology induced by the weak convergence (cf. Section 4.1 in \cite{kall}). This topology is generated by the evaluation maps $\pi_f$ such that for $m\in\mathcal{P}(E)$ it gives $m f:=\int_E f(e)\,m(de)$ for $f\in C^b_+(E)$. Here, $C^b_+(E)$ denotes the class of all bounded, continuous functions $f:\, E\to \R_+$. The space $\mathcal{P}(E)$ is metrizable and a Polish space (cf. Lemma 4.5 in \cite{kall}). We consider now the space $\mathbb{L}^0_{\G}\,(\Omega,\mathcal{P}(E))$ of all $\G$-random probability measures over $E$, i.e. all the $\G$-measurable random variables $m:\, \Omega\to \mathcal{P}(E)$. We equip this space with the topology induced by the convergence in probability. It is well-known that $m_n \xrightarrow{wP} m$ (that is, $m_n$ converges weakly in probability to $m$) if and only if any sub-sequence $N'\subset \N$ admits a further sub-sequence $N''\subset N'$ converging a.s. A further characterization of such convergence is based on the following known result.
    \begin{lemma} \label{lemma_wPconv}
        (cf. Lemma 4.8 in \cite{kall}) For any $(m_n)_{n\in\N}$ and $m$ in $\randprobspace$, $m_n  \xrightarrow{wP}  m$ if and only if
        \begin{equation*}
            m_n f=\int_E f(e)\,m_n(\cdot, de)\,  \xrightarrow{P}  m f=\int_E f(e)\,m(\cdot, de) \quad \forall\,f\in \hat{C}^b_+(E),
        \end{equation*}
        where $\hat{C}^b_+(E)$ denotes the class of all bounded, continuous functions $f:\, E\to \R_+$ with bounded support. 
    \end{lemma}
    \begin{remark}
        The condition from Lemma \ref{lemma_wPconv} above is stated taking in consideration  the function class $\hat{C}^b_+(E)$. This is due to the equivalence of the weak and vague convergence over $\mathcal{P}(E)$, where the vague convergence is generated by all evaluation maps $\pi_f$ for $f\in\hat{C}^b_+(E)$.
    \end{remark}
     
    \subsection{The mean-field game of optimal stopping} \label{subsection_MFGoptstopp} Take now $E:= [0,T]$, with the distance $d_{\rho}$ defined as $d_{\rho}(s,t):=\abs{e^{-\rho t}- e^{-\rho s}}$ for a fixed discount factor $\rho>0$, where $e^{-\rho (+\infty)}:=0$. Then, $[0,T]$ is a compact Polish space.

    Notice that, for $T<+\infty$, the distance $d_{\rho}$ is strongly equivalent to the Euclidean distance. For $T=+\infty$, any sequence diverging to $+    \infty$ in the Euclidean distance converges to $+\infty$ in $d_{\rho}$. 

    \medskip Let $h:\Omega\times [0,T]\times \randproboptstopp\to \R$ and $g:\Omega\times [0,T]\times \randproboptstopp\to \R$ be the (random) running and terminal reward functions, respectively.
    For any fixed $m\in\randproboptstopp$, the representative agent seeks to solve the optimal stopping problem
    \begin{equation}
        \sup_{\tau \in \mathcal{S} }\, \expect{ \int_0^{\tau} e^{-\rho\,t}\,h(t,m)\,dt +e^{-\rho \,\tau}g(\tau,m)  }=: \sup_{\tau \in \mathcal{S} }\, J(\tau,m),\label{eq_def_functional}
    \end{equation}
    where we use the convention $e^{-\rho \tau} g(\tau,m)=\limsup_{t\to+\infty} e^{-\rho t} g(t,m)$ on $\graffe{\tau=+\infty}.$ 
    
    At equilibrium, the given $m\in\randproboptstopp$ needs to coincide with the conditional law (given $\G$) of the optimal stopping time  $\tau^{m,\,*}$ that maximizes $J(\tau,m)$. This two-step procedure, which is common in the mean-field games literature, is formalized in the following definition of a strong mean-field equilibrium.

    \begin{definition} \label{def_mfequilibrium}
     For any $\varepsilon \geq 0$, a couple $(m^*,\tau^*)\in \randproboptstopp\times \mathcal{S}$ is a \textbf{strong $\varepsilon$-mean-field equilibrium} if 
        \begin{equation*}
            J(\tau^*,m^*)\geq \sup_{\tau\in\mathcal{S}} J(\tau,m^*) -\varepsilon \quad \text{and} \quad m^*=\mathcal{L}( \,\tau^* \,|\,\G).
        \end{equation*}
        When $\varepsilon=0$, we say that it is a \textbf{strong mean-field equilibrium}.
    \end{definition}

    \noindent Notice that this notion of equilibria resembles the one of strong equilibria in \cite{carmona2017mean}.

    \begin{remark}
        In the mean-field games' literature, the common noise is usually represented by a filtration $\mathbb{G}:=(\G_t)_{t\in[0,T]}$ and not by a single $\sigma$-algebra. However, whenever the immersion property or H-property holds (cf. Section 1.1.1 in \cite{carmona2018probabilistic} Volume II), as in the case of a Brownian motion, the conditional distribution knowing $\G_t$ at time $t\geq 0$ coincides with the conditional distribution knowing the tail $\sigma$-algebra $\G_T$. 
    \end{remark}

    Throughout the paper we will make the following assumptions on the reward processes $h(\cdot,m)$ and $g(\cdot,m)$.
    \begin{assumption} \label{hp_optstopp}
        For any fixed $m\in \randproboptstopp$,
        \begin{enumerate}
            \item[(i)] the process $(\omega,t)\to e^{-\rho t}\,g(\omega,t,m)$ is optional, of class (D) and upper semicontinuous in expectation\footnote{We recall that an optional process $Y$ is \textit{of class (D)} if  $\graffe{Y_{\tau}: \; \tau\in\mathcal{S}}$ is $\pr$-uniformly integrable. Moreover, $Y$ is \textit{upper-semicontinuous in expectation} if for any monotone sequence $(\tau_n)_{n\geq1}$ of stopping times converging a.s. to some stopping time $\tau$ one has
    \begin{equation*}
        \E \quadre{Y_{\tau}} \geq \limsup
        _{n\to +\infty} \E \quadre{Y_{\tau_n}}. 
    \end{equation*}}; 
            \item[(ii)] the process $(\omega,t)\to h(\omega,t,m)$ is progressively measurable with  
        \begin{equation*}
            \expect{\int_0^{T} e^{-\rho\,t}\abs{h(t,m)}\,dt}<+\infty.
        \end{equation*}
        \end{enumerate}
    \end{assumption}

   \smallskip
   \section{  Existence of strong randomized mean-field equilibria under continuity conditions} \label{section_optstopp_cont}

  In this Section, we build upon the result of existence of strong $\varepsilon$-mean-field equilibria established in \cite{mf-bek} in order to prove existence of a strong randomized equilibrium (cf. Definition \ref{def_relaxedNE}). In a strong randomized mean-field equilibrium, the mean-field interaction term remains adapted to the common noise $\G$, while the stopping time is randomized, as in the framework introduced in \cite{baxter_chacon}. 

   \smallskip \subsection{Existence of strong $\varepsilon$-mean-field equilibria} \label{subsection_optstopp_cont}   
   We begin by briefly presenting the result in \cite{mf-bek}, whose proof is outlined in Appendix \ref{appendix_a} for completeness. Here we slightly generalize the argument in \cite{mf-bek} so to allow also for the case $T=\infty$.
     
     \smallskip In order to solve the mean-field game of optimal stopping presented in Section \ref{subsection_MFGoptstopp}, we consider its connection to a mean-field version of Bank-El Karoui's representation problem. More precisely, for any $m\in\randproboptstopp$ we define an optional process $Y^m$ and a function $f^m:\,\Omega\times [0,T]\times \R\to\R$ as
    \begin{align*}
        Y^{m}_t&:=e^{-\rho t}g(t,m),\\
        f^m(\omega,t,l)&:=e^{-\rho t}\, (\,h(\omega,t,m)+l\,),
    \end{align*}
    for all $\omega\in\Omega,\,t\in[0,T]$. Then, under Assumption \ref{hp_optstopp}, for any fixed $m\in\randproboptstopp$ Bank-El Karoui's theorem (cf. Theorem 3 in \cite{bank-elkaroui}) yields the existence of an optional process $L^m$ such that
    \begin{equation*}
        Y^m_{\tau}= \condexp{\int_{\tau}^T\,f^m(t,\sup_{v\in[\tau,t)}L_v^m\,)\,dt}{\F_{\tau}} \quad\text{a.s.}\,\,\forall\,\tau\in\mathcal{S}.
    \end{equation*}
    Moreover, the stopping time 
    \begin{equation}
        \tau^m:=\,\inf\{t\in[0,T]:\,\sup_{v\in[0,t)}\,L_v>0\}\wedge T \label{eq_def_optimalstoppingtime}
    \end{equation}
    is the largest stopping time such that
    \begin{equation*}
       \tau^m\in\argmax_{\tau\in\mathcal{S}}\expect{ \int_0^{\tau} f^m(t,0)\,dt +Y^m_{\tau}  },
   \end{equation*}
   i.e. $\tau^m\in\argmax_{\tau\in\mathcal{S}}J(\tau,m)$ and thus solves the optimal stopping problem (\ref{eq_def_functional}) for fixed $m\in\randproboptstopp$; cf. Proposition \ref{prop_opttimes_hat_l} in the Appendix. Notice that we could have alternatively considered the hitting time
   \begin{equation*}
       \sigma^m:=\,\inf\{t\in[0,T]:\,\sup_{v\in[0,t)}\,L_v \geq 0\}\wedge T,
   \end{equation*}
   which represents the smallest stopping time that maximizes $J(\cdot,m)$ (cf. Proposition \ref{prop_opttimes_hat_l}). Hovewer, since this section  is concerned solely with an existence result, either choice realizes the purpose. A couple $(\tau^m,m)$ is now a strong mean-field equilibrium (cf. Definition \ref{def_mfequilibrium}) if $m$ is a fixed point for the map $m\to \mathcal{L}(\tau^m\,|\G)$. In order to prove the existence of fixed points, we use Schauder's theorem, which requires the compactness of the set of the mean-field interaction terms and the continuity of the map.

   \smallskip Compactness will be attained through the Prokhorov's theorem and the additional assumption that the common noise $\G$ is countably generated by a partition of $\Omega$. Indeed, in this case, there exists an identification between the space of $\G$-random probability measures and the product space $(\mathcal{P}(E))^\N$, equipped with the product weak topology. Therefore, the existence of a compact and convex subset of $\randproboptstopp$ containing all mean-field interaction terms can be established through a tightness argument.

    \begin{remark} \label{remark_commonnoise}
        The assumption that $\G$ is generated by a countable partition excludes the case of a common noise generated by a Brownian motion. However, the assumption holds  when $\G$ is generated by a random variable with countable values, or is the tail $\sigma$-algebra of an irreducible, recurrent Markov Chain where all states have finite period (cf. Theorem 5.7.3. in \cite{durrett2019probability}).
    \end{remark}

   On the other hand, the map $m\to\mathcal{L}(\tau^m\,|\,\G)$ is not continuous with respect to the weak convergence in probability. Recall that for any $m\in\randproboptstopp$, $\tau^m$ is defined as in (\ref{eq_def_optimalstoppingtime}), i.e. as an hitting time of the running supremum process $\sup_{v\in[0,t)}L^m_v=:\hat{L}^m_t$. The paths of the process $\hat{L}^m$ belong to the space $\mathbb{V}^+$ of left-continuous, non-decreasing functions $v:\,[0,T)\to\R\cup\{-\infty\}$, which we equip with the Lévy distance $d_L$. Now, Theorem 3.1 in \cite{mf-bek} establishes the continuity of the map $m\to \hat{L}^m$ with respect to the convergence in probability under $d_L$. However, the following Lemma and Remark show that  hitting times of non-decreasing processes are not continuous with respect to $d_L$, unless the limit process has strictly increasing paths.

    \begin{lemma} \label{lemma_contopttimes}
        For any functions $v_n,v$ in $\mathbb{V}^+$ and fixed $l\in\R$, let
        \begin{align*}
            \tau^n&:=\inf \graffe{t\in[0,T): \,v_n(t)>l}\wedge T \quad\forall\,n\in\N,\\
            \tau&:= \inf \graffe{t\in[0,T):\, v(t)>l}\wedge T.
        \end{align*}
        If $v$ is strictly increasing and $d_L(v_n,v)\to0$, then $ \tau^n\to\tau.$
    \end{lemma}
  
    \begin{remark} \label{remark_hittingtimes}
        The assumption of strict monotonicity of $v$ is necessary. Assume $l=0$ and choose $t_1<T$. For any $n$, let $v_n(t):=\frac{1}{n}$ for $t\in(0,t_1+\frac{1}{n})$ and $v_n(t):=t-t_1$ for $t\in[ t_1+\frac{1}{n},T)$. Also, let $v(t):=0$ for $t\in(0,t_1)$ and $v(t):=t-t_1$ for $t\in[t_1,T)$. These functions are continuous and $v_n(t)\to v(t)$ for any $t$, so $d_L(v_n,v)\to 0$ due to Proposition C1 in \cite{mf-bek}. However, $\tau^n=0$ for any $n$ and $\tau=t_1$.
    \end{remark}
    In light of Lemma \ref{lemma_contopttimes} and Remark \ref{remark_hittingtimes}, Proposition 2.15 in \cite{mf-bek} determines existence of strong $\varepsilon$-mean-field equilibria by replacing $\tau^m$ (the hitting time of $(\hat{L}^m_t)_{t\in[0,T}$) with $\tau^{m,\,\varepsilon}$ (the hitting time of the strictly increasing process $(\hat{L}^m_t+\delta_{\varepsilon}\, t)_{t\in[0,T]}$), to gain continuity of the map $m\to \mathcal{L}(  \tau^{m,\,\varepsilon}\,|\,\G)$. 

    \smallskip \begin{theorem} \label{theo_optstopp}
       (cf. Proposition 2.15 in \cite{mf-bek}) Under Assumption \ref{hp_optstopp}, assume that the process $(e^{-\rho t}g(t,m)\,)_{t\in[0,T]}$ has a.s. left upper-semicontinuous paths for all $m\in \randproboptstopp$. Moreover, assume that
       \begin{enumerate}
           \item $\G$ is generated by a countable partition of $\Omega$, i.e. there exists a countable collection of disjoint measurable sets $(A_i)_{i\in\N}$ with  $\Omega=\cup_i A_i$ and $\pr(A_i)>0$ for all $i$, such that $\G=\sigma(A_i:\,i\in\N)$;
           \item for any sequence $(m^n)_n$ and $m^{\infty}$ in $\randproboptstopp$ such that $m^n\stackrel{wP}{\to} m^{\infty}$, we have
           \begin{align}
                \lim_{n\to+\infty} \mathbb{E} \quadre{  \sup_{t\in[0,T]} e^{-\rho\,t}\abs{g(t,{m^n})-g(t,{m^{\infty}})}+\int_0^{T} e^{-\rho\,t}\abs{ h(t,{m^n})-h(t,{m^{\infty}})}\,dt  }=0.
           \label{eq_limite}
           \end{align}
       \end{enumerate}
        Then, for any $\varepsilon>0$ there exists a strong $\varepsilon$-mean-field equilibrium.
    \end{theorem}

   \smallskip\subsection{Existence of strong randomized mean-field equilibria} Theorem \ref{theo_optstopp} establishes the existence of strong $\varepsilon$-mean-field equilibria $(m^{*,\,\varepsilon},\tau^{*,\,\varepsilon})$ for any $\varepsilon>0$. However, its arguments cannot be used to prove the existence of strong mean-field equilibria. This is due to the fact that the optimal stopping times are characterized as hitting times of a process that does not necessarily have strictly increasing paths and, as noted in Remark \ref{remark_hittingtimes}, such hitting times are generally not continuous with respect to the process under the Lévy distance. This prevents the application of Schauder's fixed point theorem. A natural way now to find a strong mean-field equilibrium would be to study the convergence of  $(m^{*,\,\varepsilon},\tau^{*,\,\varepsilon})$ for $\varepsilon\to 0$. However, although the mean-field interactions $m^{*,\,\varepsilon}$ admit a converging subsequence due a compactness argument, we cannot conclude the same for the optimal stopping times. Indeed, in general, there is no guarantee that the limit of a sequence of stopping times is again a stopping time. Therefore, we weaken the definition of strong mean-field equilibria by allowing the stopping times to be randomized, in order to recover a compactness property. We refer to \cite{baxter_chacon} for the first definition and properties of randomized stopping times.

    \bigskip We consider an extended probability space.  Let $\tilde{\Omega} :=\Omega\times[0,1]$ and $(\tilde{\F_t})_t$ the filtration such that, for any $t\in[0,T]$, $\tilde{\F}_t:=\F_t\times \mathcal{B}_{[0,1]}$, where $\mathcal{B}_{[0,1]}$ is the Borel $\sigma$-algebra over $[0,1]$. Let $\tilde{\mathbb{P}}=\mathbb{P}\otimes dv$, where $dv$ is the Lebesgue measure over $[0,1]$, and $\tilde{\mathbb{E}}$ the expectation w.r.t. $\tilde{\pr}$. 
    \begin{definition}
        A random variable $\tilde{\tau}:\tilde{\Omega}\to [0,T]$ is a \textit{randomized stopping time} if $\graffe{\tilde{\tau}\leq t}\in\Tilde{\F}_t$ for any $t\in[0,T]$ and such that, for any $\omega\in\Omega$, $v\to\tilde{\tau}(\omega,v)$ is non-decreasing and left-continuous.
    \end{definition}
    Let $\tilde{\mathcal{S}}$ be the collection of all randomized stopping times. Clearly, every $\mathbb{F}$-stopping time $\tau$ belongs to $\tilde{\mathcal{S}}$, through the natural inclusion $i:\,\mathcal{S}\to \tilde{\mathcal{S}}$ such that $i(\tau)\, (\omega,v):=\tau(\omega)$ for any $(\omega,v)\in\tilde{\Omega}$. 

    \begin{remark}
        In the literature, randomized stopping times have been equivalently called mixed stopping times and are related to the concept of mixed strategies, as introduced by \cite{aumann}. One can think of randomized stopping time as a two-step randomization: first, the agent generates a realization $v$ of a uniformly distributed random variable over $[0,1]$ and then plays the corresponding strategy, i.e. the stopping time $\tau(\cdot,v)$. We also refer to \cite{touzi2002continuous} for a discussion of the different notions of randomized stopping times.
    \end{remark}

    \smallskip Over the space $\tilde{\mathcal{S}}$ we consider the Baxter-Chacon topology (see \cite{baxter_chacon}). Fix a $\sigma$-algebra $\mathcal{A}$ over $\Omega$. Then, for any sequence $\tilde{\tau}_n$ and $\tilde{\tau}$ in $\tilde{\mathcal{S}}$, $\tilde{\tau}_n $ converges to $\tilde{\tau}$ in the Baxter-Chacon topology, i.e. $\tilde{\tau}_n  \xrightarrow{BC}  \tilde{\tau}$, if and only if 
    \begin{equation*}
        \tilde{\mathbb{E}} \quadre{\,Y\,\varphi(\tilde{\tau}_n)}\to\tilde{\mathbb{E}} \quadre{\,Y\,\varphi(\tilde{\tau}_n)}\quad \quad \forall \,Y\in \mathbb{L}^1(\Omega,\mathcal{A},\mathbb{P}),\,\,\varphi\in C^b([0,+\infty]).
    \end{equation*}
    Trivially, the Baxter-Chacon convergence implies weak convergence, by choosing $Y=1$. The randomization procedure just depicted allows us to retrieve the fundamental property that any limit of randomized stopping times w.r.t. the Baxter-Chacon convergence is again a randomized stopping time, as shown in the next result.
    \begin{theorem} \label{theo_compact_BC}
        (cf. Theorem 1.5 in \cite{baxter_chacon}) The space $\tilde{\mathcal{S}}$ with the topology induced by the Baxter-Chacon convergence is compact. If, moreover, $\mathcal{A}$ is countably generated, it is sequentially compact.       
    \end{theorem}
   The randomization procedure does not modify the value of the optimal stopping problem. Indeed, for any mean-field interaction term $m\in\randproboptstopp$ and randomized stopping time $\tilde{\tau}$, recalling the definition of $J$ in (\ref{eq_def_functional}) let
    \begin{equation}
        \tilde{J}(\tilde{\tau},m):= \, \tilde{\mathbb{E}} \quadre{\int_0^{\tilde{\tau}}\,e^{-\rho t} h(t,m)\,dt+  e^{-\rho \tilde{\tau}}g(\tilde{\tau},m)}=\int_{[0,1]} J(\tilde{\tau}(\cdot,v),m)\,dv. \label{eq_def_tildeJ}
    \end{equation}
    We claim that, for any fixed $m\in\randproboptstopp$, optimizing  $\tilde{J}(\tilde{\tau},m)$ over all $\tilde{\tau}\in\tilde{\mathcal{S}}$ coincides with optimizing $J(\tau,m)$ over $\tau\in\mathcal{S}$. This allows us to work only in the setting of the relaxed stopping times henceforth.
    \begin{lemma} \label{lemma_relaxed}
        For any fixed $m$, $ \sup_{\tilde{\tau}\in\tilde{\mathcal{S}}}\tilde{J}(\tilde{\tau},m)=\sup_{\tau\in\mathcal{S}}J(\tau,m)$.
    \end{lemma} 
    \begin{proof}
        Since every $\mathbb{F}$-stopping time $\tau$ is  also a relaxed stopping time and $J(\tau,m)=\tilde{J}(\tau,m)$, clearly $\sup_{\tilde{\tau}\in \tilde{\mathcal{S}}} \tilde{J}(\tilde{\tau},m)\geq \sup_{\tau\in\mathcal{S}} J(\tau,m).$
     Moreover, for any relaxed stopping time $\tilde{\tau}$ and fixed $v\in[0,1]$, let $\tau_v(\omega):=\tilde{\tau}(\omega,v)$. We claim that $\tau_v$ is a $\mathbb{F}$-stopping time. Indeed, let $t\in[0,T]$:
     \begin{equation*}
         \graffe{\omega:\, \tau_v(\omega)\leq t}\times \graffe{v}=\graffe{ (\omega,x):\, \tilde{\tau}(\omega,x)\leq t }\,\cap\,(\Omega\times\graffe{v}) \in \F_t\times \mathcal{B}_{[0,1]}
     \end{equation*}
     and in particular $\graffe{\omega:\, \tau_v(\omega)\leq t}\in\F_t$. Thus,
     \begin{equation*}
         \tilde{J}(\tilde{\tau},m)= \int_0^1 J(\tau_v,m)\,dv\leq \int_0^1 \sup_{\tau\in\mathcal{S}} J(\tau,m)\,dv=\sup_{\tau\in\mathcal{S}} J(\tau,m). \qedhere
     \end{equation*}
     \end{proof}

    \smallskip We are now able to state our relaxed definition of a strong mean-field equilibrium (i.e. a relaxed version of Definition \ref{def_mfequilibrium}), where we allow the stopping times to be randomized but we force the mean-field interaction term to again be a $\G$-measurable random variable defined over $\Omega$.This is done by conditioning with respect to the $\sigma$-algebra $\G\times \{\emptyset,[0,1] \}$. As any random variable measurable with respect to the trivial $\sigma$-algebra is a.s. constant, any random probability measure measurable with respect to $\G\times\{\emptyset,[0,1] \} $ is in fact a $\G$-measurable random variable on $\Omega$. This formulation ensures that the strong nature of the solutions is preserved, meaning that the mean-field interaction terms remain adapted to the common noise.
    
     \begin{definition} \label{def_relaxedNE}
     A couple $(m^*,\tilde{\tau}^*)$, where $m^*\in\randproboptstopp$ and $\tilde{\tau}\in \tilde{\mathcal{S}}$, is a \textbf{strong randomized $\varepsilon$-mean-field equilibrium} if 
    \begin{equation*}
        \tilde{J}(\tilde{\tau}^*,m^*)\geq\sup_{\tilde{\tau}\in \tilde{\mathcal{S}}} \tilde{J}(\tilde{\tau},m^*)-\varepsilon  \quad \text{and} \quad m^*=\mathcal{L}( \tilde{\tau}^* \,|\,\G\times \graffe{\emptyset, [0,1]}).
    \end{equation*}
    When $\varepsilon=0$, we say that it is a \textbf{strong randomized mean-field equilibrium}. 
    \end{definition}

    \medskip As previously discussed, we will prove the existence of a strong randomized mean-field equilibrium by studying the convergence for $\varepsilon\to 0$ of the sequence of strong (non randomized) mean-field equilibria $(m^{*,\,\varepsilon},\tau^{*,\,\varepsilon})$, whose existence is ensured by Theorem \ref{theo_optstopp}. Indeed, first we will prove that there exists a couple $(m^*,\tilde{\tau}^*)$ and a suitable subsequence $\varepsilon_k\to 0 $ such that $m^{*,\,\varepsilon_k}  \xrightarrow[]{wP} m^*$ and $\tau^{*,\,\varepsilon_k}\xrightarrow[]{BC} \tilde{\tau}^*$. Secondly, we will prove that this couple is actually a strong randomized mean-field equilibrium, according to Definition \ref{def_relaxedNE}.

    \smallskip In order to accomplish that, we need here to make further mild assumptions on the structure of the reward functions $h$ and $g$. In particular, we assume they depend on the mean-field interaction term $m$ only through an auxiliary stochastic process $X^m$ with continuous paths. For $T<+\infty$, we equip the space $C([0,T])$ of continuous functions $f:[0,T]\to \R$ with the uniform norm $\norm{\cdot}_\infty$, while for $T=\infty$ we equip the space $C([0,+\infty))$ with the topology of the uniform convergence over compact sets $K$, which is metrizable with metric $\norm{\cdot}_{\infty,\,K}$ (cf. example IV.2.2 in \cite{conway1994course}).

    \smallskip \begin{theorem}
          Under Assumption \ref{hp_optstopp}, let $\G$ be generated by a countable partition of $\Omega$ and the process $(e^{-\rho t}g(t,m)\,)_{t\in[0,T]}$ have a.s.\ left upper-semicontinuous paths for any $m\in \randproboptstopp$. Moreover, assume that
          \begin{enumerate}
              \item[a)] there exist a continuous and bounded function $\hat{g}:\, [0,T]\times\R\to\R$ and a function $\hat{h}:\,[0,T]\times\,C([0,T])\to\R$ such that for any $m\in \randproboptstopp$ there exists an adapted stochastic process $X^m$ with continuous paths and
              \begin{equation*}
                  h(\omega,t,m)=\hat{h}(t,X^m(\omega)) \quad \text{and} \quad  g(\omega,t,m)=\hat{g}(\,t,X^m_t(\omega))\,\indicator{\,t<+\infty} ;
              \end{equation*}
                \item[b)] for any $m^n\stackrel{wP}{\to} m^{\infty}$ in $\randproboptstopp$, we have
                \begin{align*}
                     \mathbb{E} \Bigl[\,\norm{ X^{m^n}_t-X_t^{m^\infty}}_{\infty,\,K}  \,&+\sup_{t\in[0,T]} e^{-\rho\,t}\abs{\hat{g}(t,X^{m^n}_t)-\hat{g}(t,X^{m^{\infty}}_t)}\\
                     &+\int_0^{T} e^{-\rho\,t}\abs{ \hat{h}(t,X^{m^n})-\hat{h}(t,X^{m^{\infty}})}\,dt \, \Bigr]\to 0.
                \end{align*}
          \end{enumerate}
        Then, there exists a strong randomized mean-field equilibrium as in Definition \ref{def_relaxedNE}.

    \end{theorem}
    \begin{proof} The proof is organized in various steps.
    \vspace{0,25cm}
    
    \textit{Step a.}  Since all the assumptions of Theorem \ref{theo_optstopp} hold here, for any $\varepsilon>0$ there exists a strong (non randomized) $\varepsilon$-mean-field equilibrium $(\tau^{*,\,\varepsilon},m^{*,\,\varepsilon})$ as in Definition \ref{def_mfequilibrium}. In particular,
        \begin{equation*}
            m^{*,\,\varepsilon}\in\mathcal{H}:=\graffe{\mathcal{L}(\,\tau\,|\,\G):\,\tau\in\mathcal{S}}.
        \end{equation*}
        The set $\mathcal{H}$ is (sequentially) relatively compact w.r.t.\ the weak convergence in probability, as shown in the proof of Theorem \ref{theo_optstopp}. In particular, there exists a random probability measure $\bar{m}$ in $\randproboptstopp$ such that, up to a sub-sequence $\varepsilon_k\to 0$,
        \begin{equation*}
            m^{*,\,\varepsilon_k}\xrightarrow[]{wP} \bar{m}.
        \end{equation*}
        We define now $\mathcal{A}:=\G \,\vee\,\sigma(X^{\bar{m}})$ as the $\sigma$-algebra for the Baxter-Chacon convergence (cf. Theorem \ref{theo_compact_BC} and the text before it). Since $\mathcal{A}$ is countably generated,  due to Theorem \ref{theo_compact_BC} the set of randomized stopping times is sequentially compact. In particular, there exists a  $\bar{\tau}\in\tilde{\mathcal{S}}$ such that, up to a sub-sequence,
        \begin{equation*}
            \quad \tau^{*,\,\varepsilon_k}\xrightarrow[]{BC}\bar{\tau}.
        \end{equation*}

        \medskip \noindent \textit{Step b.} We claim that 
        \begin{equation*}
            \bar{m}=\mathcal{L} (\, \bar{\tau} \,|\,\G\times \graffe{\emptyset,[0,1]})
        \end{equation*}
        or, equivalently, that for any $f\in \hat{C}^b_+\,( [0,T])$
        \begin{equation}
            \bar{m}\,f= \tildecondexp{ f(\bar{\tau}) }{\G\times \graffe{\emptyset, [0,1]}}. \label{eq_dim_existrandom_claim}
        \end{equation}

        \smallskip Fix any $f\in \hat{C}^b_+\,( [0,T])$.  We notice that $\bar{m}\,f$ is a $\G$-measurable random variable since $\bar{m}\in\randproboptstopp$. Hence, Claim (\ref{eq_dim_existrandom_claim}) is proved if 
        \begin{equation*}
             \expect{ \indicator{G}\,\,\bar{m}\,f }= \tildeexpect{\indicator{G}\,\, f(\bar{\tau}) }, \quad \quad \forall\,\,G\in\G.
         \end{equation*}
        On the one hand, due to $m^{*,\,\varepsilon_k}\xrightarrow[]{wP}\bar{m}$ and Lemma \ref{lemma_wPconv}, 
        \begin{equation*}
              m^{*,\,\varepsilon_k}\,f=\condexp{f(\tau^{*,\,\varepsilon_k})}{\G} \xrightarrow[]{P} \bar{m}\,f,
        \end{equation*}
        and, in particular, up to a sub-sequence, the convergence may be assumed to hold almost surely. Therefore, for any fixed $G\in\G$
        \begin{equation*}
            \expect{ \indicator{G}\,\,  m^{*,\,\varepsilon_k}\,f }=\expect{\indicator{G}\,\, f(\tau^{*,\,\varepsilon_k}) }\to \expect{ \indicator{G}\,\,\bar{m}\,f },
        \end{equation*}
        due to the Dominated Convergence Theorem, as $f$ is bounded. On the other hand,
        \begin{align*}
            \expect{\indicator{G}\,\, f(\tau^{*,\,\varepsilon_k})}\xrightarrow[]{} \tildeexpect{\indicator{G}\,f(\bar{\tau})},
        \end{align*}
        due to the the definition of Baxter-Chacon convergence.  Therefore, Claim (\ref{eq_dim_existrandom_claim}) holds  for any $f\in \hat{C}^b_+\,( [0,T])$.

        \bigskip \noindent \textit{Step c.} Let $X^k:=X^{m^{*,\,\varepsilon_k}}$ and $\bar{X}:=X^{\bar{m}}$. We claim that 
        \begin{equation}
            \lim_{k\to+\infty} \abs{\tildeexpect{ \,e^{-\rho\, \tau^{*,\,\varepsilon_k}} g(\tau^{*,\,\varepsilon_k},m^{*,\,\varepsilon_k})-\,e^{-\rho\,\bar{\tau}}g(\bar{\tau},\bar{m}) }}= 0. \label{eq_dim_existence_1}
        \end{equation}
        We take inspiration from the proofs of Propositions 10 and 12 in \cite{coquet}.

        \medskip Since $m^{*,\,\varepsilon_k}\xrightarrow[]{wP}\bar{m}$, assumption b) in the statement of the theorem yields that $X^k$ converges a.e.\ to $\bar{X}$ w.r.t.\ the uniform convergence over compacts.  In particular, both the sequences $X^k$ and $\tau^{*,\,\varepsilon_k}$ converge weakly and thus are tight. The sequence $(X^k,\tau^{*,\,\varepsilon_k})$ is tight for the product topology. Now, in order to prove our Claim (\ref{eq_dim_existence_1}), it is sufficient to show the convergence of the finite-dimensional distributions. 

        \smallskip Let $t_1<\dots<t_N$ and $f:\R^N\to\R$ and $\phi:\R\to\R$ continuous bounded functions. Then,
        \begin{align*}
           &\abs{\,\tildeexpect{\,f(X^k_{t_1},\dots, X^k_{t_N})\,\phi(\tau^{*,\,\varepsilon_k})\,}-\tildeexpect{\,f(\bar{X}_{t_1},\dots, \bar{X}_{t_N})\,\phi(\bar{\tau})}\,}\\
           &\leq \abs{ \,\tildeexpect{ \,\tonde{f(X^k_{t_1},\dots, X^k_{t_N})-f(\bar{X}_{t_1},\dots, \bar{X}_{t_N})} \,\phi(\tau^{*,\,\varepsilon_k})\,}\,} \\
           &\quad\quad+ \abs{\, \tildeexpect{f(\bar{X}_{t_1},\dots, \bar{X}_{t_N})\, \tonde{ \phi(\tau^{*,\,\varepsilon_k})-\phi(\bar{\tau}) }}  } \\
           &\leq ||\phi||_\infty \, \expect{ \,|f(X^k_{t_1},\dots, X^k_{t_N})-f(\bar{X}_{t_1},\dots, \bar{X}_{t_N})\,|}\\
           &\quad\quad +\abs{\, \tildeexpect{f(\bar{X}_{t_1},\dots, \bar{X}_{t_N})\, \tonde{ \phi(\tau^{*,\,\varepsilon_k})-\phi(\bar{\tau}) }} }
        \end{align*}
        The last term on the very right-hand side of the last displayed equation converges to zero by definition of Baxter-Chacon convergence, since $f(\bar{X}_{t_1},\dots, \bar{X}_{t_N})\in\mathbb{L}^1(\sigma(\bar{X}))$. Moreover, we know that $\mathbb{E} [\, \sup_{i=1,\dots,N} \,|X^k_{t_i}-\bar{X}_{t_i}|\,]\to 0$ and the continuous mapping theorem yields the convergence to zero of the first term of the very right-hand side of the last displayed equation. Through a density argument, we can then expand this result to any continuous bounded function $f:\R^{N+1}\to \R$ and thus
        \begin{equation*}
            (X^k_{t_1},\dots, X^k_{t_N},\tau^{*,\,\varepsilon_k})\,\xrightarrow[]{w}\, (\bar{X}_{t_1},\dots, \bar{X}_{t_N},\bar{\tau}).
        \end{equation*}
        This fact and the tightness of the sequence $(X^k,\tau^{*,\,\varepsilon_k})$ imply $(X^k,\tau^{*,\,\varepsilon_k})\xrightarrow[]{w} (\bar{X},\bar{\tau}).$

        \medskip Now, thanks to Skorokhod's representation theorem, there exists a probability space $(\Omega',\mathbb{F}',\pr')$ supporting $Y^k,\,\sigma^k,\,\bar{Y},\bar{\sigma}$ such that 
        \begin{equation*}
            Y^k \sim X^k,\,\, \sigma^k\sim\tau^{*,\,\varepsilon_k},\,\,\bar{Y}\sim \bar{X},\,\, \bar{\sigma}\sim \bar{\tau} \quad \text{and} \quad (Y^k,\sigma^k)\xrightarrow[]{a.s.} (\bar{Y},\bar{\sigma}).
        \end{equation*}
        Hence,
        \begin{align}
            &\abs{\tildeexpect{ \,e^{-\rho\, \tau^{*,\,\varepsilon_k}} g(\tau^{*,\,\varepsilon_k},m^{*,\,\varepsilon_k})-\,e^{-\rho\,\bar{\tau}}g(\bar{\tau},\bar{m}) }} \nonumber \\
            &=\abs{ \mathbb{E}' \quadre{ e^{-\rho\, \sigma^k}\, \hat{g}(\sigma^k,Y^k_{\sigma^k}) \,\indicator{\sigma^k<+\infty}- e^{-\rho\,\bar{\sigma} } \hat{g}(\bar{\sigma},\bar{Y}_{\sigma})\,\indicator{\bar{\sigma}<+\infty} } } \nonumber\\
            &\leq \mathbb{E}' \quadre{\abs{   e^{-\rho\, \sigma^k}\, \hat{g}(\sigma^k,Y^k_{\sigma^k}) \,\indicator{\sigma^k,\,\bar{\sigma}<+\infty}- e^{-\rho\,\bar{\sigma} } \hat{g}(\bar{\sigma},\bar{Y}_{\sigma})\,\indicator{\bar{\sigma}<+\infty} } } \nonumber\\
            &\,\,\,\,\,\,\,+\,  \mathbb{E}' \quadre{ e^{-\rho\, \sigma^k}\,\abs{    \hat{g}(\sigma^k,Y^k_{\sigma^k}) }\,\indicator{\sigma^k<+\infty,\,\bar{\sigma}=+\infty}} \nonumber\\
            &\leq \mathbb{E}'\quadre{ \abs{e^{-\rho\, \sigma^k}\, \hat{g}(\sigma^k,Y^k_{\sigma^k}) -e^{-\rho\,\bar{\sigma} } \hat{g}(\bar{\sigma},\bar{Y}_{\sigma})}\,\indicator{\bar{\sigma}<+\infty} } \nonumber\\
            &\,\,\,\,\,\,\,+\,  \mathbb{E}' \quadre{ e^{-\rho\, \sigma^k}\,\abs{    \hat{g}(\sigma^k,Y^k_{\sigma^k}) }\,\indicator{\sigma^k<+\infty,\,\bar{\sigma}=+\infty}}, \label{eq_dim_existence_2}
        \end{align}
        where the last inequality is due to the fact that $\graffe{\bar{\sigma}<+\infty}\subseteq \graffe{\sigma^k<+\infty}$ definitely in $k$, since $\sigma^k\to \bar{\sigma}$ a.s. Now, Claim (\ref{eq_dim_existence_1}) is proved if the expectations on the very right hand-side of (\ref{eq_dim_existence_2}) converge to zero as $k\to+\infty$. On the one hand, on the set $\graffe{\bar{\sigma}=+\infty}$, since $\sigma^k\to +\infty$ a.s. the exponential term $e^{-\rho\, \sigma^k}$ goes to zero, while $\hat{g}$ is bounded. Hence, the second expectation of the right-hand side of (\ref{eq_dim_existence_2}) tends to zero due to the Dominated Convergence Theorem. On the other hand, a.e.\ on the set $\graffe{ \bar{\sigma}<+\infty}$ there exists a compact set $K$ in $[0,+\infty)$ such that $\bar{\sigma},\sigma^k\in K$ definitely in $k$. Then, since  $\bar{Y}$ has continuous paths, we conclude that
        \begin{align*}
            | \,Y^k_{\sigma^k}-\bar{Y}_{\bar{\sigma}}\, |&\leq | \, Y^k_{\sigma^k}-\bar{Y}_{\sigma^k}\,| +|\,\bar{Y}_{\sigma^k}-\bar{Y}_{\sigma}  \,|\leq \sup_{t\in K} |\,Y^k_t-\bar{Y}_t\,| +|\,\bar{Y}_{\sigma^k}-\bar{Y}_{\sigma}  \,|\to 0,
        \end{align*}
        i.e. $Y^k_{\sigma^k}$ converges a.s. to $\bar{Y}_{\bar{\sigma}}$. Then, the first expectation of  the right-hand side of (\ref{eq_dim_existence_2}) tends to zero again due to the Dominated Convergence Theorem.

        \bigskip\noindent \textit{Step d.} Finally, we claim that $\bar{\tau}$ maximizes $\tilde{J}(\tilde{\tau}, \bar{m})$ over all randomized stopping times. Due to Lemma \ref{lemma_relaxed}, we know that it is sufficient to check over the subset of non-randomized stopping times.

        \medskip Since $(\tau^{*,\,\varepsilon_k}, m^{*,\,\varepsilon_k})$ are strong $\varepsilon_k$-mean-field equilibria, we have that
        \begin{equation}
            J(\tau^{*,\,\varepsilon_k},  m^{*,\,\varepsilon_k})\geq \sup_{\tau\in\mathcal{S}} J(\tau,  m^{*,\,\varepsilon_k})-\varepsilon_k. \label{eq_dim_existence}
        \end{equation}
        Consider now the right-hand side of (\ref{eq_dim_existence}). For any fixed $k$, we know there exists an optimal stopping time $\sigma^k$ for $J(\cdot,m^{*,\,\varepsilon_k})$, which is sub-optimal for $J(\cdot, \bar{m})$. From these observations, we deduce that
        \begin{align*}
            &\sup_{\tau\in\mathcal{S}} J(\tau, {m^{*,\,\varepsilon_k}})-\sup_{\tau\in\mathcal{S}} J(\tau, \bar{m})\leq J(\sigma^k, m^{*,\,\varepsilon_k})-J(\sigma^k,\bar{m})\\
            &\leq \expect{ e^{-\rho\,\sigma^k} \tonde{\, g(\sigma^k,m^{*,\,\varepsilon_k})-g(\sigma^k,\bar{m} ) }+\int_0^{\sigma^k} e^{-\rho\,t} \tonde{h(t, m^{*,\,\varepsilon_k})-h(t,\bar{m})} \,dt }\\
            &\leq \expect{ \sup_{t\in[0,T]} e^{-\rho t} \abs{g(t,m^{*,\,\varepsilon_k})-g(t,\bar{m})} +\int_0^T e^{-\rho t}\, \abs{h(t,m^{*,\,\varepsilon_k})-h(t,\bar{m})}\,dt  }.
        \end{align*}
        With a similar argument (using now an optimal stopping time for $J(\cdot,\bar{m})$), we conclude that
        \begin{align*}
            &\abs{\sup_{\tau\in\mathcal{S}} J(\tau, {m^{*,\,\varepsilon_k}})-\sup_{\tau\in\mathcal{S}} J(\tau, \bar{m})} \\
            &\leq \expect{ \sup_{t\in[0,T]} e^{-\rho t} \abs{g(t,m^{*,\,\varepsilon_k})-g(t,\bar{m})} +\int_0^T e^{-\rho t}\, \abs{h(t,m^{*,\,\varepsilon_k})-h(t,\bar{m})}\,dt  }\to 0,
        \end{align*}
        where the convergence to zero follows by assumption since $m^{*,\,\varepsilon_k}\xrightarrow[]{wP}\bar{m}$. Therefore, the right-hand side of (\ref{eq_dim_existence}) tends to $ \sup_{\tau\in\mathcal{S}} J(\tau,\bar{m}).$

        \medskip We turn now to the left-hand side of (\ref{eq_dim_existence}):
        \begin{align*}
            \abs{ J(\tau^{*,\,\varepsilon_k},m^{*,\,\varepsilon_k})-\tilde{J}(\bar{\tau},\bar{m}) }\leq \,&   \abs{ \tildeexpect{e^{-\rho\,\tau^{*,\,\varepsilon_k}} g(\tau^{*,\,\varepsilon_k},m^{*,\,\varepsilon_k})-e^{-\rho\,\bar{\tau}}g(\bar{\tau},\bar{m} )}  } \\
            &+\abs{\tildeexpect{ \int_0^{\tau^{*,\,\varepsilon_k} }\, e^{-\rho t} \tonde{h(t,m^{*,\,\varepsilon_k})-h(t,\bar{m})}\,dt} }\\
            &+\abs{ \tildeexpect{\int_0^{\tau^{*,\,\varepsilon_k}}e^{-\rho t}h(t,\bar{m})\,dt- \int_0^{\bar{\tau}} e^{-\rho t}h(t,\bar{m})\,dt} }.
        \end{align*}
        The first expectation in the right-hand side goes to zero as $k\to+\infty$ by Step c. Moreover, as for the second expectation,
        \begin{equation*}
           \abs{\tildeexpect{ \int_0^{\tau^{*,\,\varepsilon_k} }\, e^{-\rho t} \tonde{h(t,m^{*,\,\varepsilon_k})-h(t,\bar{m})}\,dt}} \leq\tildeexpect{\int_0^T e^{-\rho t} \abs{ h(t,m^{*,\,\varepsilon_k})-h(t,\bar{m}) }\,dt},
        \end{equation*}
        which converges to zero by assumption. Lastly,
        \begin{align*}
            &\lim_k  \abs{ \tildeexpect{ \int_0^T \, e^{-\rho t} h(t,\bar{m})\, \tonde{ \indicator{t\leq \tau^{*,\,\varepsilon_k}}-\indicator{t\leq \bar{\tau}} }\,dt } }\\
            &\leq \lim_k \int_0^T \abs{\tildeexpect{  e^{-\rho t} h(t,\bar{m})\, \tonde{ \indicator{t\leq \tau^{*,\,\varepsilon_k}}-\indicator{t\leq \bar{\tau}} }}}\,dt\\
            &=\int_0^T \lim_k  \abs{\tildeexpect{  e^{-\rho t} h(t,\bar{m})\, \tonde{ \indicator{t\leq \tau^{*,\,\varepsilon_k}}-\indicator{t\leq \bar{\tau}} }}}\,dt,
        \end{align*}
        where the inequality stems from Fubini's theorem and the equality is due to the Dominated Convergence Theorem. Indeed,
        \begin{equation*}
             \abs{\tildeexpect{  e^{-\rho t} h(t,\bar{m})\, \tonde{ \indicator{t\leq \tau^{*,\,\varepsilon_k}}-\indicator{t\leq \bar{\tau}} }}} \leq 2\, \expect{e^{-\rho t}\abs{h(t,\bar{m})}},
        \end{equation*}
        which is integrable over $[0,T]$ due to Assumption \ref{hp_optstopp}.  Moreover, for almost every $t\in [0,T]$ the random variable $ e^{-\rho t} h(t,\bar{m}) $ belongs to $\mathbb{L}^1(\sigma(X^{\bar{m}}))$. Therefore,  by definition of Baxter-Chacon convergence and by an approximation procedure of indicator functions with functions in $ C([0,T])$ (cf. Lemma 1.37 in \cite{kallenberg1997foundations}),  we have that
        \begin{equation*}
            \lim_k \abs{\tilde{\mathbb{E}} \quadre{e^{-\rho t}h(t,\bar{m}) \, \tonde{ \indicator{t\leq \tau^{*,\,\varepsilon_k}}-\indicator{t\leq \bar{\tau}} }} }=0
   \end{equation*}
    for a.e. $t$. Therefore, the left-hand side of (\ref{eq_dim_existence}) tends to $\tilde{J}(\bar{\tau},\bar{m})$ and we finally have
    \begin{equation*}
        \tilde{J}(\bar{\tau},\bar{m})\geq \sup_{\tau\in\mathcal{S}} J(\tau, \bar{m}). \qedhere
    \end{equation*}
    \end{proof}

    \smallskip
    \section{  Existence and Comparative Statics of strong equilibria under monotonicity conditions}\label{section_optstopp_mon}
   In this Section we look for strong mean-field equilibria in the sense of Definition \ref{def_mfequilibrium} in a setting with an additional order structure. Our aim is to establish an order of the extremal points of the sets of equilibria with respect to ordered sets of reward functions. To accomplish that, we will exploit the characterization of the smallest and largest optimal stopping times as hitting times of an auxiliary Bank-El Karoui's representation problem (cf. Proposition \ref{prop_opttimes_hat_l} in the Appendix).
   
   In particular, we consider the natural order $\R$ and over $\mathcal{P}( [0,T])$ the first-order stochastic dominance $\leq_p$. Thus, for any $\mu_1,\mu_2\in\mathcal{P}([0,T])$, $\mu_1\leq_p \mu_2$ if and only if
    \begin{equation*}
        \int_{[0,T]} \phi(x,t)\, d\mu_1(t)\leq\int_{[0,T]} \phi(t)\, dm_2(x,t) 
    \end{equation*}
    for any bounded increasing measurable function $\phi:\,[0,T]\to \R $. Due to Theorem 2 in \cite{kamae}, $(\mathcal{P}([0,T]),\leq_p)$ is a partially ordered Polish space. This partial order may be extended to $\randproboptstopp$: with some abuse of notation, we say that $m^1\leq_p m^2$ if and only if $m^1 (\omega)\leq_p m^2(\omega)$ a.s. 

    \medskip As we have previously discussed in Section \ref{subsection_optstopp_cont}, for any fixed $m\in\randproboptstopp$ the largest optimal stopping time $\tau^m$ for $J(\cdot,m)$ may be characterized through the solution $L^m$ of a related Bank-El Karoui's representation problem. Moreover, any couple $(\tau^m,m)$ where $m$ is a fixed point of the map $m\to\mathcal{L}(\tau^m\,|\,\G)$ is a strong mean-field equilibrium. In this setting, we wish to use Tarski's theorem to prove the existence of fixed points. Hence, we require the set of all relevant mean-field interaction terms to be a complete lattice and the monotonicity of the map $m\to \mathcal{L}(\, \tau^m\,|\,\G)$. Results of existence of mean-field equilibria in similar order settings are present in the literature, e.g. \cite{carmona2017mean}, \cite{submodularity}, \cite{dianetti2023unifying},  \cite{mf-bek} and  \cite{possamai2023mean}.

    \begin{theorem} (cf. Proposition 2.20 in \cite{mf-bek}) \label{theo_optstopp_2}
    Under Assumption \ref{hp_optstopp}, assume that for any $m^1,m^2\in V$ s.t. $m^1 \leq_p m^2$, one has
        \begin{align*}
               &e^{-\rho \,t} \tonde{g(t,m^1)-g(t,m^2)} \,\,\text{is a supermartingale},\\
               &h(\cdot,m^1)\leq h(\cdot,m^2).
        \end{align*}
    Then, there exists a strong mean-field equilibrium according to Definition \ref{def_mfequilibrium}.
    \end{theorem}
    \noindent The proof of this result is outlined in Appendix \ref{appendix_a}.

    \bigskip We wish now to run a comparative statics analysis over the sets of strong mean-field equilibria with respect to the reward functions. We consider two different optimal stopping problems for $i=1,2$ with given running reward functions $h^1,h^2$ and terminal reward functions $g^1,g^2$, where both couples satisfy Assumption \ref{hp_optstopp}. We make the following additional requirement.
    \begin{assumption}\label{hp_comparativestatics}
        Let $h^1(\cdot)\geq h^2(\cdot)$ and $e^{-\rho t} \,(g^1(t,m)-g^2(t,m) \,)$ be a submartingale for any fixed $m\in\mathbb{L}^0_{\G}(\Omega,\mathcal{P}([0,T]))$.
    \end{assumption}

     \smallskip \noindent For $i=1,2$  consider the two optimal stopping problems
    \begin{equation*}
        \sup_{\tau\in\mathcal{S}} J^i(\tau,  m):=\sup_{\tau\in\mathcal{S} } \,\expect{ \int_0^{\tau} e^{-\rho\, t}\,h^i(t,m)\,dt+ e^{-\rho \tau} g^i(\tau,m) }.
    \end{equation*}
    Thanks to Theorem \ref{theo_optstopp_2}, we known that strong mean-field equilibria exist for both problems. In particular, they may be found as the fixed points of the maps
    \begin{equation*}
         T^i(m):= \mathcal{L}(\, \tau^{m,\,i} \,|\,\G),
    \end{equation*}
    where $\tau^{m,\,i}:=\inf \{t:\, \hat{L}^{m,\,i}>0 \}\wedge T$ and $\hat{L}^{m,\,i}$ solve an appropriate auxiliary Bank-El Karoui's representation problem. With a similar argument, one can prove that the fixed points of the maps
    \begin{equation*}
        S^i(m):= \mathcal{L}(\, \sigma^{m,\,i} \,|\,\G),
    \end{equation*}
    where $\sigma^{m,\,i}:=\inf \{t:\, \hat{L}^{m,\,i}\geq 0 \}\wedge T$, are also strong mean-field equilibria. Notice that $\tau^{m,\,i}$ are the largest optimal stopping times for $J^i(\cdot,m)$, while $\sigma^{m,\,i}$ are the smallest optimal stopping times (see Proposition \ref{prop_opttimes_hat_l} in the Appendix).

    \smallskip \noindent 
    \begin{lemma} \label{lemma_comparativestatics}
        Under Assumption \ref{hp_comparativestatics} and the assumptions of Theorem \ref{theo_optstopp_2}, it holds: 
        \begin{enumerate}
            \item[(i)] for any $m$, $T^1(m)\geq T^2(m)$ and $S^1(m)\geq S^2(m)$;
            \item[(ii)] for any fixed point ${m^{2,*}}$ of $T^2$, there exists a larger fixed point ${m^{1,*}}$ of $T^1$;
            \item[(iii)]   for any fixed point  ${m^{1,*}}$ of $S^1$, there exists a smaller fixed point ${m^{2,*}}$ of $S^2$;
        \end{enumerate}
    \end{lemma}

    \smallskip\noindent\textit{Proof of (i).} Fix $m\in\randproboptstopp$. We claim that $\hat{L}^{m, \,1}_{\tau}\leq \hat{L}^{m,\,2}_{\tau}$ a.s. for any stopping time $\tau$. Then, since they are both optional processes we may conclude that 
    $\hat{L}^{m, \,1}_t\leq \hat{L}^{m,\,2}_t$ for any time $t$ a.s.  and by the order-inverting property of the hitting times we have that $T^1(m)\geq T^2(m)$ and $S^1(m)\geq S^2(m)$.

    \smallskip To prove our claim, consider any stopping time $\tau$. Due to Theorem 1 in \cite{bank-elkaroui}, we can assume without loss of generality that the solutions ${L}^{m,\,i}$ have upper right continuous paths, so that
    \begin{equation*}
        L^{m,\,i}_{\tau}=\essinf_{\sigma\in \mathcal{S},\,\sigma >\tau} \,\,l^{m,\,i}_{\tau,\,\sigma},
    \end{equation*}
    where $l^{m,\,i}_{\tau,\,\sigma}$ is the unique $\F_{\tau}$-measurable random variable such that
    \begin{equation}
        \condexp{\int_{\tau}^{\sigma}  f^{m,\,i}(t,l^{m,\,i}_{\tau,\,\sigma})\,dt }{\F_{\tau}}= \condexp{Y^{m,\,i}_{\tau}-Y^{m,\,i}_{\sigma}}{\F_{\tau}}. \label{eq_comparative}
    \end{equation}
     Our claim is now proved if $l^{m,\,1}_{\tau,\,\sigma} \leq l^{m,\,2}_{\tau,\,\sigma}$ for any $\sigma >\tau$. Assume, by contradiction, that
    there exists a stopping time $\sigma$ such that the set $A:=\, \graffe{l^{m,\,1}_{\tau,\,\sigma} > l^{m,\,2}_{\tau,\,\sigma}}$ has positive probability.
    We recall that $Y^{m,\,i}$ and $f^{m,\,i}$ are defined as in the proof of Theorem \ref{theo_optstopp_2}, i.e.
    \begin{align*}
        Y^{m,\,i}_t&:=e^{-\rho\,t}g^i(t,m)-\condexp{e^{-\rho\,T}\,g^i(T,m)}{\F_t}, \\
        f^{m,\,i}(t,l)&:=e^{-\rho t} \,(\,h^i(t,m)+l).
    \end{align*}
    Since the term $\condexp{e^{-\rho\,T}\,g^i(T,m)}{\F_t}$ is a martingale, we can rewrite the right-hand side of (\ref{eq_comparative}) as
    \begin{equation*}
        \condexp{  e^{-\rho\,\tau}\, g^i(\tau,m)-e^{-\rho\,\sigma}g^i(\sigma,m) \,}{\F_{\tau}}.
    \end{equation*}
    Now, since by Assumption \ref{hp_comparativestatics} $e^{-\rho t}\, (g^1(t,m)-g^2(t,m))$ is a submartingale, we have that
    \begin{equation*}
        e^{-\rho \tau}\, (g^1(\tau,m)-g^2(\tau,m))\leq \condexp{e^{-\rho \sigma}\, (g^1(\sigma,m)-g^2(\sigma,m))}{\F_{\tau}},
    \end{equation*}
    which, if plugged into (\ref{eq_comparative}), gives us
    \begin{align*}
        \condexp{\int_{\tau}^{\sigma}  e^{-\rho t} (h^1(t,m)+l^{m,\,1}_{\tau,\,\sigma} \,)\,dt  }{\F_{\tau}} \leq \condexp{\int_{\tau}^{\sigma}  e^{-\rho t} (h^2(t,m)+l^{m,\,2}_{\tau,\,\sigma} \,)\,dt  }{\F_{\tau}}.
    \end{align*}
    Then, since $A\in\F_{\tau}$,
    \begin{equation*}
       \condexp{\indicator{A}\,\int_{\tau}^{\sigma}  e^{-\rho t} \tonde{ \,h^1(t,m)-h^1(t,m)+l^{m,\,1}_{\tau,\,\sigma}-l^{m,\,2}_{\tau,\,\sigma} }  \,dt }{\F_{\tau}}\leq 0,
    \end{equation*}
    but over $A$ the integrand is strictly positive, since $h^1\geq h^2$ by Assumption \ref{hp_comparativestatics}. Hence a contradiction.
    
    \bigskip \noindent\textit{Proof of (ii)}  Let $V':=\,\graffe{ m\in\randproboptstopp:\,\,m\geq_p m^{*,\,2} }$, which is a complete lattice. Due to $(i)$, for any $m\in V'$ we have that 
    \begin{equation*}
        T^1(m)\geq_p T^1(m^{*,\,2})\geq_p T^2(m^{*,\,2})=m^{*,\,2},
    \end{equation*}   
    since $m^{*,\,2}$ is a fixed point for $T^2$. Thus, the restriction $T^1|_{V'}:\, V'\to V'$ is an increasing map over a complete lattice and therefore admits a fixed point $m^{*\,1}$ due to Tarski's theorem. 

     \bigskip \noindent\textit{Proof of (iii)}  The statement is proved with the same argument of Step $(ii)$, after noticing that for any $m\leq_p m^{*,\,1} $
     \begin{equation*}
         S^2(m)\leq_p S^2( m^{*,\,1})\leq_p S^1( m^{*,\,1})= m^{*,\,1}.
     \end{equation*}
    \qedhere

    \smallskip
    \appendix

    \section{}\label{appendix_b}
    We briefly present the mean-field version of the Bank-El Karoui's representation theorem as introduced in \cite{mf-bek}.
    
    \medskip We consider a more general setting than the one introduced in Section \ref{section_problem}. More precisely, here $E$ is a general non-empty Polish space, while in the rest of paper $E$ coincides with $[0,T]$. Again, $\mathbb{L}^0_{\G}\,(\Omega,\mathcal{P}(E))$ is the space of all $\G$-random probability measures over $E$, equipped with the topology induced by the weak convergence in probability. Moreover, let $\mathbb{V}^+$ be the space of all non-decreasing, left-continuous functions $v:[0,T)\to \R \,\cup\,\graffe{-\infty}$ such that $v(0)=-\infty$ and $v$ is real-valued on $(0,T)$. We equip $\mathbb{V}^+$ with the Lévy distance $d_L$, which makes it a Polish space (cf. Proposition C1 in \cite{mf-bek}). We recall that the convergence with respect to $d_L$ coincides with the pointwise convergence on all times $t$ of continuity for the limit function (cf.\  Proposition C1 in \cite{mf-bek}).

    \medskip For any $m\in \randprobspace$, let $Y^m$ be an optional process and $f^m:\Omega\times [0,T]\times\R\to\R$ a function that satisfy the the following assumption.
    
   \begin{assumption} \label{hp_mfbek}
        \begin{enumerate}
            \item[i)] the process $Y^m$ is real-valued, optional, of class (D), upper semicontinuous in expectation with $Y^m_T=0$ a.s.; 
            \item[ii)] for any fixed $(\omega,t)\in\Omega\times [0,T]$, the map $l\to f^m(\omega,t,l)$ is continuous and strictly increasing from $-\infty$ to $+\infty$. For any fixed $l\in\R$, the process $(\omega,t)\to f^m(\omega,t,l)$ is progressively measurable with 
        \begin{equation*}
            \expect{\int_0^{T} \abs{f^m(t,l)}\,dt}<+\infty.
        \end{equation*}
        \end{enumerate}
   \end{assumption}
    Lastly, let $\Psi: \Omega\times \mathbb{V}^+\to E$ be a measurable function. We may now define the solution to the mean-field version of Bank-El Karoui's representation problem.
    \begin{definition} \label{def_beksolution}
        A couple $(L,m)$, where $L$ is an optional process and $m\in \randprobspace$, solves the mean-field version of Bank-El Karoui's representation problem if 
        \begin{align}
            &Y^m_{\tau}=\condexp{ \int_{\tau}^T \,f^m(t, \sup_{v\in [\tau,t)}\,L_v)\,dt }{\F_{\tau}} \quad \text{a.s.} \quad \forall\,\tau\in\mathcal{S}; \label{def_mfbek_1}
            \\ &m=\mathcal{L}(\, \Psi(\hat{L}) \,|\,\G), \label{def_mfbek_2}
         \end{align}
         where $\hat{L}_t:=\sup_{v\in [0,t)} L_v$ for any $t\in[0,T)$.
    \end{definition}
    \begin{remark}
        This definition is slightly different that the corresponding one in \cite{mf-bek}. In particular, the consistency condition lacks the dependence on an underlying process $X^m$.
    \end{remark}

    For each fixed $m\in \randprobspace$, Assumption \ref{hp_mfbek} and Bank-El Karoui's representation theorem (cf. Theorem 3 in \cite{bank-elkaroui}) yield the existence of an optional process $L^m$ which solves the representation problem w.r.t. $Y^m$ and $f^m$, i.e. such that (\ref{def_mfbek_1}) holds. Moreover, the running supremum process $\hat{L}^m$  has paths in $\mathbb{V}^+$ and it is unique up to indistinguishability, so that the right hand term of (\ref{def_mfbek_2}) is well-defined. Now, a couple $(L^m,m)$ solves the mean-field version of Bank-El Karoui's representation problem if $m$ satisfies the fixed point equation 
    \begin{equation}
        m=\mathcal{L} (\,\Psi(\hat{L}^m)\,|\,\mathcal{G}\,).\label{eq_fixedpoint}
    \end{equation}

    \bigskip In a first setting, we wish to use Schauder's theorem to ensure the existence of a solution to (\ref{eq_fixedpoint}). Hence, we require the compactness of the space of all the mean-field interaction terms and the continuity of the map $m\to \mathcal{L} (\,\Psi(\hat{L}^m)\,|\,\mathcal{G}\,)$. While the first prerequisite will be assumed, the second condition will stem from the continuity of the function $\Psi$ and the stability of the solution $\hat{L}^m$ w.r.t. $Y^m$ and $f^m$ (cf. Theorem 3.1 in \cite{mf-bek}).
     \begin{theorem} \label{theo_mfbek_1} (cf. Theorem 2.4 in \cite{mf-bek})
        Under Assumption \ref{hp_mfbek}, let for any $m\in\mathbb{L}^0_{\G}(\Omega,\mathcal{P}(E))$, $Y^m$ have almost surely left upper-semicontinuous paths. Moreover, 
      \begin{enumerate}
            \item[a)] $\G$ is generated by a countable partition of $\Omega$;
            \item[b)] for any $\omega$, the map $\mathbf{l}\to \Psi(\mathbf{l})$ is continuous;
            \item[c)] there exists a compact convex subset $K$ s.t. $\{ \mathcal{L}(\Psi(\hat{L})|\G):\,\,  \hat{L}\in \mathbb{L}^0_{\mathbb{F}}(\Omega,\mathbb{V}^+) \} \subseteq K$;
            \item[d)] for any $l\in\R$ and $m^n,m^{\infty}$ in $K$ such that $m^n\stackrel{wP}{\to}m^{\infty}$, then 
            \begin{equation*}
                \lim_{n\to+\infty} \expect{  \sup_{t\in[0,T]} \abs{Y^{m^n}_t-Y^{m^{\infty}}_t}\,+ \int_0^T \abs{ f^{m^n}(t,l)-f^{m^{\infty}}(t,l)}dt  }=0.
            \end{equation*}
        \end{enumerate}
        Then, there exists a solution to mean-field Bank-El Karoui's representation problem.
    \end{theorem}

    \medskip In the second setting, we wish to apply Tarski's fixed point theorem to establish existence of a fixed point to Equation (\ref{eq_fixedpoint}). This result requires the additional structure given by partial orders. In particular, it requires the space of all the mean-field interaction terms to be a complete lattice and the monotonicity of the map $m\to \mathcal{L} (\,\Psi(X^m,\hat{L}^m)\,|\,\mathcal{G}\,)$. Hence, we suppose that $E$ is a partially ordered Polish space, with partial order $\leq_E$ and we consider the partial order $\leq_p$ given by the a.s.\ first order stochastic dominance over $\randprobspace$. Lastly, let $\leq_e$ be the partial order on $\mathbb{V}^+$ such that for  any $v_1,v_2$, $v_1\leq_e v_2$ if and only if $v_1(t)\leq v_2(t)\,\,\forall\,t\in[0,T)$.

   \begin{theorem} \label{theo_mfbek_2} (cf. Theorem 2.8 in \cite{mf-bek})
    Under Assumption \ref{hp_mfbek}, assume that
    \begin{enumerate}
        \item[a)]
        for any $l^1\leq_e l^2$, we have $ \Psi(\omega,l^1)\geq_E \Psi(\omega,l^2)$;
    \item[b)] there exists a complete lattice $K$ in $\mathbb{L}^0_{\G}(\Omega,\mathcal{P}(E))$ such that 
    \begin{equation*}
        \graffe{ \mathcal{L}(\Psi(\hat{L})|\G):\quad \hat{L}\in \mathbb{L}^0_{\F}(\Omega,\mathbb{V}^+) } \subseteq K;
    \end{equation*}
    \item[c)] for any $m^1,m^2\in K$, $m^1\leq_p m^2$ implies that
    \begin{equation*}
        Y^{m^1}-Y^{m^2} \,\,\text{is a supermartingale,}\quad \quad  f^{m^1}(\cdot)\leq f^{m^2}(\cdot).
    \end{equation*}
    \end{enumerate}
    Then, there exists a solution $(L,m)$ to the mean-field representation problem.
\end{theorem}
    \begin{remark}
        The assumptions of Theorem 2.8 in \cite{mf-bek} are here slightly modified and adapted to our setting. In particular, here we assume that $f^m$ is order-preserving w.r.t. $m$, while in \cite{mf-bek} $f^m$ is assumed to be order-inverting.
    \end{remark}

    Lastly, an essential result for our paper is the connection between Bank-El Karoui's representation problem and an optimal stopping problem. This property was thoroughly investigated in \cite{bank_follmer}.
     \begin{proposition}
        \label{prop_opttimes_hat_l}
        For any fixed mean-field interaction term $m$, let $L^m$ be the solution of the representation problem w.r.t. $Y^m$ and $f^m$. Then, for any $l\in \R$        the stopping times
        \begin{align*}
            \sigma_{\,l}^m&:=\, \inf \graffe{t\in [0,T]: \, \hat{L}^m_t\geq l}\wedge T \\
            \tau^m_{\,l}&:=\, \inf \graffe{t\in [0,T]: \, \hat{L}^m_t> l} \wedge T
        \end{align*}
        are the smallest and biggest solution to the optimal stopping problem 
        \begin{equation*}
            \sup_{\tau\in\mathcal{S}} \expect{Y^m_{\tau}+\int_0^{\tau}f^m(t,l)\,dt}.
        \end{equation*}   
    \end{proposition}

    \section{} \label{appendix_a}

    \subsection{Proofs of the results in Section \ref{subsection_optstopp_cont}}

     \noindent\\ \textbf{Proof of Lemma \ref{lemma_contopttimes}.}  Due to Proposition C1 in \cite{mf-bek}, Lévy convergence implies pointwise convergence of $v_n(t)$ to $v(t)$ for any point of continuity $t$ of $v$ in $(0,T)$. Since $v$ is increasing, it has at most a countable set of jump discontinuities, so the set $\Pi$ of continuity points is dense in $(0,T)$. 
        
        \smallskip Assume $\tau<T$. Then for any fixed $\delta >0$, there exist a time $t_1\in\Pi\cap(\tau,\tau+\delta)$  due to the density of $\Pi$. In particular, 
        \begin{equation*}
            \lim_{n\to+\infty} v_n(t_1)=v(t_1)>l
        \end{equation*}
        and thus definitely $v_n(t_1)>l$, which implies that $\tau^n\leq t_1$. Similarly, due to the density of $\Pi$ and the strict monotonicity of $v$, there exists a time $t_2\in\Pi\cap(\tau-\delta,\tau)$ such that $v(t_2)<l.$ Repeating the previous argument, we conclude that $\tau^n\geq t_2$ and thus defintely
        \begin{equation*}
           \tau-\delta< \tau^n<\tau+\delta.
        \end{equation*}              

        \smallskip Assume $\tau=T$, then clearly $v(t)\leq l$ for any $t\in[0,T)$. Assume by contradiction that there exists a subsequence $n_k$ and a $M<T$ such that $\tau^{n_k}<M$ for any $k$. By density, there exists a time $s\in \Pi\cap (M,T)$: then $v_{n_k}(s)>l$ and 
        \begin{equation*}
            l\leq\lim_{k\to+\infty} v_{n_k}(s)=v(s)\leq l,
        \end{equation*}
        i.e. $v(t)=l$ for any $t\in[s,T)$, which contradicts the strict monotonicity of $v$.
    \qed

    \bigskip \medskip \noindent \textbf{Proof of Theorem \ref{theo_optstopp}.} We just give an outline of our modifications to the proof of Proposition 2.15 in \cite{mf-bek}. In particular, we extend their proof to allow also for the case $T=\infty$. For a more detailed discussion, we refer to the original in \cite{mf-bek}.

    \smallskip \noindent \textit{Step a.} Fix $\varepsilon>0$ and let     
    \begin{equation*}
        \delta_{\varepsilon}:=\frac{\varepsilon}{3\int_0^{T}te^{-\rho t}\,dt}<+\infty.
    \end{equation*}
     For any $m\in V$, let
    \begin{align*}
            Y^m_t&:=e^{-\rho\,t}g(t,m)-\condexp{e^{-\rho\,T}\,g(T,m)}{\F_t},\\
              f^m(t,l)&:=e^{-\rho t} \,(\,h(t,m)+l)\\
             \Psi(\textbf{l})&:= \tau^{ \textbf{l},\,\varepsilon }=\inf \graffe{t:\,\, \textbf{l}_t+\delta_{\varepsilon} \,t > 0 }\wedge T.
         \end{align*}
    Due to Assumption \ref{hp_optstopp}, for any $m\in\randproboptstopp$ the couple $(Y^m, f^m)$ satisfies the assumption of the original Bank-El Karoui's representation theorem, i.e.\ Assumption \ref{hp_mfbek}. 

    \medskip \noindent\textit{Step b.} The assumptions of Theorem \ref{theo_mfbek_1}, one of the mean-field versions of the Bank-El Karoui's representation theorem, hold. The first and the last clearly hold due to our assumptions. For the continuity of $\Psi$, let $\mathbf{l}^n,\mathbf{l}$ in $\mathbb{V}^+$ s.t. $d_L(\mathbf{l}^n,\mathbf{l})\to 0$. Then, $\tau^{\mathbf{l}^n,\,\varepsilon}\to\tau^{\mathbf{l},\,\varepsilon}$ due to the strict monotonicity of $\mathbf{l}+\delta_{\varepsilon}\,id$ and Lemma \ref{lemma_contopttimes}. Lastly, let $K:=\, \overline{  conv( {\mathcal{H}})}$, where $conv({\mathcal{H}})$ denotes the convex hull and 
    \begin{equation*}
             \mathcal{H}:=\graffe{ \mathcal{L}(\,\tau\,|\,\G):\,\,\tau\in\mathcal{S} }.
    \end{equation*}
    Then,  $K$ is non-empty, convex and compact, due to Prokhorov's theorem and the fact that $[0,T]$ is compact.

    \medskip\noindent\textit{Step c.} Due to Theorem \ref{theo_mfbek_1}, there exists a solution $(L^{*,\,\varepsilon},m^{*,\,\varepsilon})$ to the mean-field Bank-El Karoui's representation problem. We define the stopping time
    \begin{equation*}
        \tau^{*,\,\varepsilon}:=\inf \graffe{t\geq 0:\, \hat{L}^{*,\,\varepsilon}_t+\delta_{\varepsilon} \,t> 0 }\wedge T.
    \end{equation*}
    Then, by definition of solution to the mean-field version of Bank-El Karoui's problem (cf. Definition \ref{def_mfbek_1}), we have that
    \begin{equation*}
        m^{*,\,\varepsilon}=\mathcal{L}( \, \tau^{*,\,\varepsilon}|\,\G),
    \end{equation*}
    i.e.\ the consistency condition of a strong $\varepsilon$-mean-field equilibrium according to Definition \ref{def_mfequilibrium}. To prove that $( \tau^{*,\,\varepsilon}, m^{*,\,\varepsilon})$ is a strong $\varepsilon$-mean-field equilibrium, it is sufficient to prove that 
      \begin{equation}
            \sup_{\tau\in\mathcal{S}} J(\tau,m^{*,\,\varepsilon})-J(\tau^{*,\,\varepsilon},m^{*,\,\varepsilon}) \leq\varepsilon.\label{eq_dimtheo_extra}
    \end{equation}
    Due to Proposition \ref{prop_opttimes_hat_l}, we know that the left-hand side of the last inequality coincides with $J(\sigma^{*,\,\varepsilon},m^{*,\,\varepsilon})-J(\tau^{*,\,\varepsilon},m^{*,\,\varepsilon})$, where
    \begin{equation*}
        \sigma^{*,\,\varepsilon}=\graffe{t\geq 0:\, \hat{L}^{*,\,\varepsilon}_t> 0 }\wedge T.
    \end{equation*}
    Then, using Bank-El Karou's representation, we can rewrite the left-hand side as
    \begin{align*}
             \abs{ \sup_{\tau\in\mathcal{S}}J(\tau,m^{*,\,\varepsilon})-J(\tau^{*,\,\varepsilon},m^{*,\,\varepsilon}) }\leq \,\,&\expect{\int_0^T e^{-\rho t} \abs{\hat{L}^{*,\,\varepsilon}_t\vee 0 \,-\,(\hat{L}^{*,\,\varepsilon}_t+\delta_{\varepsilon}\,t)\vee 0} \,dt}\\
             &+\, \expect{\int_0^T e^{-\rho t} \abs{(\hat{L}^{*,\,\varepsilon}_t+\delta_{\varepsilon}\,t)\vee 0\,-\, \indicator{t>\tau^{*,\,\varepsilon}}\, \hat{L}^{*,\,\varepsilon}_t} \,dt}\\
             &+ \expect{\int_0^T e^{-\rho t} \,\indicator{t>\tau^{*,\,\varepsilon}} \abs{ \hat{L}^{*,\,\varepsilon}_t \,-\, \sup_{v\in[\tau^{*,\,\varepsilon},t)}L^{*,\,\varepsilon}_v } \,dt} \\
             &\leq 3\,\delta_{\varepsilon} \int_0^Tte^{-\rho t}\,dt=\varepsilon.            
        \end{align*}
    The last inequality results from the fact that all the absolute value terms are bounded from above by $\delta_{\varepsilon}\,t$. This is trivial for the first two, while for the last fix a time $t>\tau^{*,\,\varepsilon}$. By definition of $\tau^{*,\,\varepsilon}$ as the first hitting time of $\hat{L}^{*,\,\varepsilon}_t+\delta_{\varepsilon}\,t$, for any $s<\tau^{*,\,\varepsilon}$ one has $L^{*,\,\varepsilon}_s\leq -\delta_{\varepsilon}\,s\leq 0$ while $\sup_{v\in[\tau^{*,\,\varepsilon},t) }L^{*,\,\varepsilon}_v\geq -\delta_{\varepsilon}\,\tau^{*,\,\varepsilon} $. Combining them, we get $\sup_{v\in[\tau^{*,\,\varepsilon},t) }L^{*,\,\varepsilon}_v\geq \hat{L}^{*,\,\varepsilon}_{\tau^{*,\,\varepsilon}}-\delta_{\varepsilon}\,\tau^{*,\,\varepsilon}.$ Therefore, since $\hat{L}^{*,\,\varepsilon}_t=\,\hat{L}^{*,\,\varepsilon}_{\tau^{*,\,\varepsilon}}\,\vee \, \sup_{v\in[\tau^{*,\,\varepsilon},t)}L^{*,\,\varepsilon}_v$, the term $|\hat{L}^{*,\,\varepsilon}_t \,-\, \sup_{v\in[\tau^{*,\,\varepsilon},t)}L^{*,\,\varepsilon}_v |$ is bounded by $\delta_{\varepsilon}\,\tau^{*,\varepsilon}$ and thus by $\delta_{\varepsilon}\,t$.

    \bigskip \subsection{Proofs of the results Section \ref{section_optstopp_mon}}

    \noindent\\ \textbf{Proof of Theorem \ref{theo_optstopp_2}.} \textit{Step a.} For any $m\in \randproboptstopp$, let
        \begin{align*}
            Y^m_t&:=e^{-\rho\,t}g(t,m)-\condexp{e^{-\rho\,T}\,g(T,m)}{\F_t},\\
            f^m(t,l)&:=e^{-\rho t} \,(\,h(t,m)+l),\\
            \Psi(\omega,\textbf{l})&:=\tau^{ \textbf{l} }=\inf\graffe{t\geq 0: \mathbf{l}_t>0}
        \end{align*}
        Due to Assumption \ref{hp_optstopp}, the couple $(Y^m, f^m)$ satisfies the assumption of the original Bank-El Karoui's representation theorem, i.e.\ Assumption \ref{hp_mfbek}. 

        \medskip \textit{Step b.} The assumptions of Theorem \ref{theo_mfbek_2}, one of the mean-field versions of the Bank-El Karoui's representation theorem, hold. Indeed, the monotonicity of $\Psi$ results from the fact that the first hitting time $\tau^l$ is decreasing w.r.t. the paths. We define 
        \begin{equation*}
            K:=\,\graffe{\mathcal{L}(\, \tau\, |\,\G):\,\,\tau\in\mathcal{S}},
        \end{equation*}
        which is a complete lattice. Finally, last assumption of Theorem \ref{theo_mfbek_2} is a consequence of the last assumption in the statement of this theorem and the fact that
        \begin{equation*}
            \condexp{e^{-\rho\,T}\,g(T,m^1)}{\F_t}-\condexp{e^{-\rho\,T}\,g(T,m^2)}{\F_t}
        \end{equation*}
        is a martingale for any $m_1,m_2\in\randproboptstopp$.

        \medskip \textit{Step c.} Due to Theorem \ref{theo_mfbek_2}, there exists a solution $(L^*,m^*)$ to the mean-field Bank-El Karoui's representation problem. We define
        \begin{equation*}
            \tau^*:=\Psi(\hat{L}^*)=\inf\graffe{t\geq 0: \hat{L}^*_t>0}.
        \end{equation*}
        Then, by definition of solution to the mean-field version of Bank-El Karoui's problem (cf. Definition \ref{def_mfbek_1}), we have that
        \begin{equation*}
            m^*=\mathcal{L}( \tau^* \,|\,\G),
        \end{equation*}
        i.e.\ the consistency condition of a strong mean-field equilibrium according to Definition \ref{def_mfequilibrium}. The optimality of $\tau^*$ for $J(\cdot,m^*)$ results from Proposition \ref{prop_opttimes_hat_l}.\qedhere
     
    \smallskip    
    
\bigskip 
\textbf{Acknowledgements.}  
Financial support by the German Research Foundation (DFG) [RTG 2865/1 - 492988838] is gratefully acknowledged.

{\normalsize

\bibliographystyle{siam}
\bibliography{biblio}
}


 \end{document}